\documentclass{amsart}
\usepackage[utf8]{inputenc}
\usepackage{amsthm}
\usepackage{amsmath}
\usepackage{tikz}
\usetikzlibrary{arrows}

\usepackage{enumerate}   
\usepackage{amssymb}
\usepackage{fancyhdr} 
\usepackage[left=3.5cm,right=3.5cm,top=3.5cm,bottom=3.5cm]{geometry}
\usepackage{xcolor}

\usepackage{graphicx}
\usepackage{ esint }
\newtheorem{theorem}{Theorem}[section]
\newtheorem{proposition}[theorem]{Proposition}
\newtheorem{corollary}[theorem]{Corollary}
\newtheorem{lemma}[theorem]{Lemma}
\newtheorem{definition}[theorem]{Definition}
\newtheorem{example}[theorem]{Example}
\newtheorem*{theoremV}{Voronin's theorem}
\newtheorem*{theoremA}{Theorem A}

\theoremstyle{remark}
\newtheorem{question}[theorem]{Question}

\newcommand{\veps}{\varepsilon}
\newcommand{\CC}{\mathbb C}
\newcommand{\RR}{\mathbb R}

\newcommand{\NN}{\mathbb N}
\newcommand{\ZZ}{\mathbb Z}

\newcommand{\dwa}{{\mathcal D}_{\mathrm{w.a.}}}
\newcommand{\ldens}{\underline{\mathrm{dens}}}
\newcommand{\ddens}{{\mathcal D}_{\mathrm{dens}}}
\DeclareMathOperator{\sign}{sign}

\usepackage{url}
\usepackage[hidelinks]{hyperref}
\hypersetup{
	colorlinks=true,
	linkcolor=cyan,
	filecolor=cyan,
	citecolor =cyan,      
	urlcolor=cyan,
}
\author{Frédéric Bayart}
\address{Laboratoire de Math\'ematiques Blaise Pascal UMR 6620 CNRS, Universit\'e Clermont Auvergne, Campus universitaire des C\'ezeaux, 3 place Vasarely, 63178 Aubi\`ere Cedex, France.}
\email{frederic.bayart@uca.fr}
\author{Athanasios Kouroupis}
\address{Department of Mathematical Sciences, Norwegian University of Science and Technology (NTNU), 7491 Trondheim, Norway}
\email{athanasios.kouroupis@ntnu.no}

\title[Universality]{On universality of general Dirichlet series}
\date{}
\begin{document}
	\begin{abstract}
		In the present work, we establish sufficient conditions for a Dirichlet series induced by general frequencies to be universal with respect to vertical translations. Our results can be applied to known universal objects such as Hurwitz zeta functions and also can provide new examples of universal Dirichlet series including the alternating prime zeta function $\sum_{n\geq1}(-1)^np_n^{-s}$.
	\end{abstract}
	\maketitle
	
	\section{Introduction}
	
	The study of the universal properties of Dirichlet series goes back to 1975 with the seminal work of Voronin on the Riemann zeta function \cite{VORONIN}. Voronin's theorem says:
	
	\begin{theoremV}
		Let $K$ be a compact subset of $\{1/2<\Re e(s)<1\}$ with connected complement, let $f$ be a nonvanishing function continuous on $K$ and holomorphic in the interior of $K$. Then 
		$$\ldens\left\{\tau\geq 0:\ \sup_{s\in K}|\zeta(s+i\tau)-f(s)|<\veps\right\}>0$$
		where $\ldens(A)$ denotes the lower density of $A\subset\RR_+$, that is
		\[\ldens(A)=\liminf\limits_{T\rightarrow\infty}\frac{1}{T}\int\limits_{0}^{T}\mathbf{1}_A(t)\,dt.\]
	\end{theoremV}
	
	We will give sufficient conditions for a Dirichlet series to approximate similarly to the Riemann zeta function every holomorphic function on suitable domains. Recall that a Dirichlet series is a function of the form:
	$$D(s)=\sum_n a_n e^{-\lambda_n s},$$
	where $(a_n)\subset\CC^{\mathbb N}$ and $(\lambda_n)$ is a frequency, namely an increasing sequence of nonnegative real numbers tending to $+\infty$. The case $(\lambda_n)=(\log n)$ corresponds to ordinary Dirichlet series. 
	
	Let us introduce the following definitions: let $\Omega_1\subset\Omega\subset\mathbb C$ be two domains such that $\Omega_1+i\tau\subset\Omega_1$ for all $\tau>0$, and, for all compact sets $K\subset\Omega,$ there exists $\tau>0$ with $K+i\tau\subset\Omega_1$. Let $D:\Omega_1\to\mathbb C$ be holomorphic.
	We say that $D$ is {\bf universal} in $\Omega$ if for all compact subsets $K$ of $\Omega$ with connected complement,
	for all nonvanishing functions $f:K\to\Omega,$ continuous on $K$ and holomorphic in the interior of $K,$
	$$\ldens\left\{\tau\geq 0:\ \sup_{s\in K}| D(s+i\tau)-f(s)|<\veps\right\}>0.$$
	We say that $D$ is {\bf strongly universal} if the restriction that $f$ is nonvanishing can be eased. Equivalently, a Dirichlet series is strongly universal if it can approximate locally uniformly, via vertical translations, every complex polynomial in $\Omega$. 
	
	Since Voronin's work, the area of universality gained popularity. Many authors studied aspects of (strong) universality for various classes of Dirichlet series. The survey paper \cite{Matsu15} provides a thorough examination of the subject up to 2015.
	
	The first author in \cite{Bayrearrangement} improving the work of \cite{LSS03} on strong universality of general Dirichlet series obtained the following result.
	
	\begin{theoremA}
		Let $P\in\mathbb R[X]$ with $\deg(P)=d\geq 1$ and $\lim_{+\infty}P=+\infty$, let $Q\in\mathbb R[X]$ with $\deg(Q)=d-1$,
		let $\omega\in\mathbb R\backslash 2\pi\ZZ$ and let $\kappa\in\mathbb R$. 
		Assume moreover that the sequence $(\log(P(n))_{n\geq 1}$ is $\mathbb Q$-linearly independent. Then the Dirichlet series $D(s)=\sum_{n\geq 1}Q(n)(\log n)^\kappa e^{i\omega n}(P(n))^{-s}$
		is strongly universal in $\{(2d-1)/2d<\Re e(s)<1\}$.  
	\end{theoremA}
	
	This generalizes the case of the Lerch zeta function (see \cite{GaLau02}) when $Q(n)=1,$ $\kappa=0$ and
	$P(n)=n+\alpha$ with $\alpha$ transcendental.
	
	\medskip
	
	To prove Theorem A, a key step is to evaluate the square moments of $D$. This uses classical techniques of harmonic analysis like the method of non-stationary phase (see the Van der Corput type lemma \ref{lem:vandercorput}) as well as a Hilbert type inequality (see Lemma \ref{lem:MV}). These tools require some regularity assumptions on the sequence $(\lambda_n)$ and one cannot apply them to the simplest frequency of $\mathbb Q-$linearly independent numbers, namely $(\log(p_n))$ where $(p_n)$ is the increasing sequence of prime numbers. 
	
	It seems that in the recent literature, an old result of Landau \cite{Landau09}, which gives under mild conditions an estimate of the square moments of a convergent Dirichlet series, has been forgotten. Using this result, we are able to obtain, with an elementary proof, the following far-reaching extension of Theorem A (the definitions of the abscissas of convergence of a Dirichlet series are given in Subsection \ref{sec:abscissas}).
	
	%One of the objectives of this article is to study potential cases of universal Dirichlet series for which the methods of 
	%\cite{Bayrearrangement} are not applicable, see for example \cite[Question~6.8]{Bayrearrangement}. The first one concerns the frequencies $(\lambda_n)$.
	%The simplest example of a frequency $(\lambda_n)$ such that $(\log(\lambda_n))$ is $\mathbb Q$-linearly independent 
	%is probably the sequence $(p_n)$ of prime numbers. However, this sequence is not regular enough
	%to be handled by the methods of \cite{Bayrearrangement}. The main barrier for this problem is that the sequence of primes is not regular enough to estimate its partial sums using the classical 
	%techniques from harmonic analysis like the method of non-stationary phase/ Van der Corput type lemmas \ref{lem:vandercorput} or decoupling \cite{BOU17}. 
	%We give a partial answer to this question.
	%We are able to prove strong universality for Dirichlet series with type of frequencies and with similar (and even slightly relaxed) conditions on the coefficients.
	%However, the universal property is only shown on a smaller strip.
	
	\begin{theorem}\label{thm:main}
		Let $(\lambda_n)$ be a frequency, let $D(s)=\sum_{n\geq 1}a_n e^{-\lambda_n s}$ be a Dirichlet series. Assume that the frequency $(\lambda_n)$ is $\mathbb Q$-linearly independent, satisfies (WLC) and that for all $\alpha,\,\beta>0,$ there exist $C>0$ and $x_0\geq 1$ such that, for all $x\geq x_0,$ 
		\begin{equation}\label{eq:thmmain}
		\sum_{\lambda_n \in \left[x,x+\frac\alpha{x^2}\right]}|a_n |\geq Ce^{(\sigma_a(D)-\beta)x}.
		\end{equation}
		Then $D$ is strongly universal in $\{(\sigma_c(D)+\sigma_a(D))/2<\Re e(s)<\sigma_a(D)\}$.
	\end{theorem}
	
	We postpone to Section \ref{sec:landau} the definition of the condition (WLC). We just mention that it is a weak property of separation of the sequence $(\lambda_n)$ which is satisfied if $\lambda_n=\log(P(n))$ or $\lambda_n=\log(P(p_n))$ where $P$ is a polynomial satisfying $\lim_{+\infty}P=+\infty$.
	
	Of course, the property \eqref{eq:thmmain} may seem unclear and hard to testify. But, will allow us to provide a plethora of examples of strongly universal Dirichlet series.
	\begin{corollary}\label{exten}
		The Dirichlet series $D(s)=\sum_{n\geq 1}a_n[P(n)]^{-s}$ is strongly universal in the strip $\{(\sigma_c(D)+1)/2<\Re e(s)<1\},$ assuming the following:
		\begin{itemize}
			\item $P$ is a real polynomial of degree  $d\geq 1$ and $\lim_{+\infty}P=+\infty$.
			\item The sequence $(a_n)$ satisfies
			$$\lim_{n\to+\infty}\frac{\log |a_n|}{\log n}=d-1.$$
			\item The frequency $(\log(P(n))$ is $\mathbb Q$-linearly independent.
		\end{itemize}
	\end{corollary}
	
	\begin{corollary}\label{exprime}
		The Dirichlet series $D(s)=\sum_{n\geq 1}a_n[P(p_n)]^{-s}$ is strongly universal in the strip $\{(\sigma_c(D)+1)/2<\Re e(s)<1\},$ assuming the following:
		\begin{itemize}
			\item $P$ is a real polynomial of degree  $d\geq 1$ and $\lim_{+\infty}P=+\infty$.
			\item The sequence $(a_n)$ satisfies
			$$\lim_{n\to+\infty}\frac{\log |a_n|}{\log n}=d-1.$$
			\item The frequency $(\log(P(p_n))$ is $\mathbb Q$-linearly independent.
		\end{itemize}
	\end{corollary}
	As a corollary, we give a positive answer to the Question~6.8 posed in \cite{Bayrearrangement}. 
	
	\begin{corollary}
		The Dirichlet series $\sum_{n\geq 1} (-1)^n p_n^{-s}$ is strongly universal in the critical strip $\{\frac 12<\Re e(s)<1\}.$
	\end{corollary}
	
	Observe that for the examples coming from \cite{Bayrearrangement} or from Theorem \ref{thm:main}, the Dirichlet series itself converges in its strip of universality.
	This does not cover the case of the Riemann zeta function or that of the Hurwitz zeta functions $\sum_n (n+\alpha)^{-s}$, $\alpha$ transcendental,
	which have a pole at $1$ and are known to be universal in $\{\frac 12<\Re e(s)<1\}$. We extend those results to a large class of general Dirichlet series,
	even if $1$ is a branching point and not a pole.
	
	In what follows we denote by $\mathbb C_{\sigma}$ the half-plane $\{\Re e(s)>\sigma\}$
	and by $\mathbb C_\sigma^+$ its restriction to the complex numbers of positive imaginary part, $\{s\in\mathbb C:\ \Re e(s)>\sigma,\ \Im m(s)>0\}$.
	
	\begin{theorem}\label{thm:main2}
		Let $P,\, Q\in\mathbb R[X]$ be polynomials of degree $d\geq 1$ and $d-1$, respectively. If $\lim_{+\infty}P=+\infty$ and the sequence $(\log(P(n)))$ is $\mathbb Q$-linearly independent, then the Dirichlet series 
		$$D(s)=\sum_n Q(n)(\log(n))^\kappa [P(n)]^{-s},\qquad\kappa\in\mathbb R$$ 
		admits a holomorphic continuation to $\mathbb C_{1-\frac 1d}^+\cup\CC_1$ 
		and even to $\mathbb C_{1-\frac 1d}\backslash\{1\}$ if $\kappa$ is a nonneggative integer. 
		Moreover, it is \textbf{strongly universal} in the strip $\left\{(2d-1)/d<\Re e(s)<1\right\}.$
	\end{theorem}
	
	\subsection*{Organisation of the paper}
	\begin{itemize}
		\item In Section \ref{sec:preliminaries} we go through some preliminary results and definitions.
		\item In Section \ref{sec:landau}, which is mostly expository, we give a complete account on Landau's theorem, under slightly more general assumptions. As an application, we provide an elementary proof for the finiteness of the square moments of the zeta function in the critical strip. 
		\item In Section \ref{sec:examples} we give several examples of strongly universal Dirichlet series and the proof of Theorem \ref{thm:main} and its corollaries.
		\item  In Section \ref{5} we study universal Dirichlet series with a singularity (or a branching point).
		\item  In the last Section \ref{sec:randommodels}, we discuss the expected universal properties of a class of random Dirichlet series.
	\end{itemize}

	\subsection*{Notation}
	Throughout the paper, if $f,g : E\to \mathbb R$ are two functions defined on the same set $E$, the notation $f\lesssim g$ will mean that there exists some constant $C > 0$ such that $f \leq Cg$ on $E$.
	\subsection*{Funding}
	A. Kouroupis is partially supported by the Onassis Foundation - Scholarship ID: F ZT 037-1/2023-2024.
	
	\section{Preliminaries}\label{sec:preliminaries}
	
	\subsection{Abscissas of convergence} \label{sec:abscissas}
	
	To a Dirichlet series $D=\sum_{n=1}^{+\infty} a_n e^{-\lambda_n s}$ we will associate three abscissas, its abscissa of convergence, 
	$$\sigma_c(D):=\inf\left\{\Re e(s):\ \sum_n a_n e^{-\lambda_n s}\textrm{ converges}\right\},$$
	its abscissa of absolute convergence 
	$$\sigma_a(D):=\inf\left\{\sigma\in\RR:\ \sum_n |a_n|e^{-\lambda_n\sigma}\textrm{ converges}\right\},$$
	and also
	$$\sigma_2(D):=\inf\left\{\sigma\in\RR:\ \sum_n |a_n|^2e^{-2\lambda_n\sigma}\textrm{ converges}\right\}.$$
	It is well-known that $D=\sum_n a_ne^{-\lambda_n s}$ converges in the half-plane $\CC_{\sigma_c(D)}$  and that it defines a holomorphic function there. It is also straightforward to check that $\sigma_2(D)\leq (\sigma_c(D)+\sigma_a(D))/2.$
	
	In what follows, we will assume that a Dirichlet series $D=\sum_n a_ne^{-\lambda_n s}$ has finite abscissa of convergence.
	
	\subsection{How to prove universality}
	
	Let us introduce two definitions from \cite{Bayrearrangement}.
	
	\begin{definition}\label{def:dwa}
		Let $\sigma_0\in\mathbb R$. We say that a Dirichlet series $D(s)=\sum_{n}a_ne^{-\lambda_n s}$  belongs to the class $\dwa(\sigma_0),\, \sigma_0\geq \sigma_2(D)$, provided:
		\begin{itemize}
			\item[(a)] It extends holomorphically to $\mathbb C_{\sigma_0}^+\cup\CC_{\sigma_c(D)}$.
			
			\item[(b)] For all $\sigma_2>\sigma_1>\sigma_0$, 
			$$\sup_{\sigma\in[\sigma_1,\sigma_2]}\sup_{T>0}\frac 1T\int_1^T |D(\sigma+it)|^2dt<+\infty.$$
			\item[(c)] The sequence $(\lambda_n)$ is $\mathbb Q$-linearly independent. 
		\end{itemize}
	\end{definition}
	Note that for a function $D$ as above, the submean value property implies that it is of order $1/2$ uniformly in $\{\Re e(s)\geq\sigma_1\}$ for all $\sigma_1>\sigma_0$, i.e,  there exist $C,\,t_0>0$ such that
	\[|D(\sigma+it)|\leq C t^{\frac{1}{2}}, \qquad\sigma\geq \sigma_1,\qquad t\geq t_0.\] 
	
	Hadamard three-lines theorem (see \cite{TIT58}) implies that the order is a strictly decreasing function of $\sigma\in(\sigma_c(D),\sigma_a(D))$. 
	If we further assume that $D$ has a holomorphic extension to $\mathbb{C}_{\sigma_0}$, then $\sigma_0\geq \sigma_c(D)$ (see \cite{TIT58}).
	
	\begin{definition}\label{def:ddens}
		We say that a Dirichlet series $D=\sum_n a_n e^{-\lambda_n s}$ belongs to $\ddens$ provided  for all $\alpha,\,\beta>0$, there exist $C>0$ and $x_0\geq 1$ such that, for all $x\geq x_0$, 
		$$\sum_{\lambda_n \in \left[x,x+\frac\alpha{x^2}\right]}|a_n |\geq Ce^{(\sigma_a(D)-\beta)x}.$$
	\end{definition}
	
	The main interest of introducing these definitions is the following theorem (see \cite{Bayrearrangement}).
	
	\begin{theorem}\label{thm:bayrearrangement}
		Let $D$ be a Dirichlet series and let $\sigma_0>\sigma_2(D).$ Assume that $D\in\dwa(\sigma_0)\cap\ddens$.
		Then $D$ is strongly universal in the strip $\{\sigma_0<\Re e (s)<\sigma_a(D)\}$. 
	\end{theorem}
	It should be pointed out that Definition \ref{def:dwa} in \cite{Bayrearrangement} mentions the whole half-plane $\CC_{\sigma_0}$ and not the quarter-plane as here. However, this does not change anything for the proofs. The key points are that the half vertical lines $\sigma+it,\ t>0,\ \sigma>\sigma_0,$ are contained in $\CC_{\sigma_0}^+$ and that for any compact set $K$ included in the strip $\{\sigma_0<\Re e (s)<\sigma_a(D)\},$ there exists $\tau>0$ such that $K+i\tau\subset \CC_{\sigma_0}^+.$
	
	The main difficulty will be to verify that the square moments of a Dirichlet series are bounded, that is condition (b) of Definition \ref{def:dwa}.
	
	\subsection{Two lemmas to estimate exponential sums}
	
	We shall need two inequalities that have been widely used in this context. 
	The first one deals with exponential sums and is due to Montgomery and Vaughan
	(see \cite{MV74}).
	
	\begin{lemma}\label{lem:MV}
		Let $(a_n)$ be a sequence of complex numbers such that $\sum_n |a_n|^2<+\infty$. Let $(\lambda_n)$ be a sequence of real numbers 
		and set $\theta_n:=\inf_{m\neq n}|\lambda_n-\lambda_m|>0$ for every $n$. Then 
		$$\int_0^T \left|\sum_n a_ne^{i\lambda_n t}\right|^2dt=T\sum_n |a_n|^2+O\left(\sum_n \frac{|a_n|^2}{\theta_n}\right)$$
		where the $O$-constant is absolute.
	\end{lemma}
	
	We also need the following classical inequality for exponential sums, which goes back to J. G. Van der Corput 
	(see \cite[Lemma 11.5]{BM09}).
	
	\begin{lemma}\label{lem:vandercorput}
		Let $a<b$ and let $f,\,g:[a,b]\to\mathbb R$ be two functions of class $\mathcal C^2.$ Assume that
		\begin{itemize}
			\item $f'$ is monotonic with $|f'|<1/2$;
			\item $g$ is positive, nonincreasing and convex.
		\end{itemize}
		Then 
		$$\sum_{n=a}^b g(n)e^{2\pi i f(n)}=\int_a^b g(u)e^{2\pi i f(u)}du+O(g(a)+|g'(a)|)$$
		where the $O$-constant is absolute.
	\end{lemma}
	
	\subsection{The incomplete Gamma function/ Prym's function}
	We will make a short presentation and we refer the interested reader to \cite{Abra,Olver,Temme}.
	For $\Re e(a)>0$ and $\Re e(z)>0$, we define the incomplete Gamma function $\Gamma(a,z)$ by
	$$\Gamma(a,z)=\int_z^{+\infty}t^{a-1}e^{-t}\,dt.$$
	For fixed $z$, as in the classical case, it has a meromorphic extension in $\mathbb{C}$ with simple poles at the nonpositive integers. This can be easily obtained from the recurrence relation:
	$$\Gamma(a+1,z)=a\Gamma(a,z)+z^a e^{-z}.$$
	For a fixed value of $a,$ $\Gamma$ admits a holomorphic extension (its principal branch) to $\mathbb C\backslash \mathbb R_-$
	and even to $\mathbb C$ when $a$ is a positive integer. When $a$ is not a nonpositive integer, this follows for instance from the relation
	$$\Gamma(a,z)=\Gamma(a)(1-z^{a-1}\gamma^*(a,z)),$$
	where the function $\gamma^*$ is entire in both $a$ and $z$.
	When $a$ is a nonpositive integer, this follows from the corresponding statement for $a=0$ (in that case, the incomplete Gamma function
	is also called the exponential integral).
	
	For this principal branch (and for a fixed $a$), we have the estimation
	\begin{equation}\label{eq:ineggamma}
	\Gamma(a,z)= e^{-z}\int_{0}^{+\infty}e^{-u}\left(z+u\right)^{a-1}\,du=O(z^{a-1}e^{-z}),
	\end{equation}
	as $|z|\to+\infty.$
	
	\subsection{A remark on \cite{Bayrearrangement}}
	In \cite[Theorem~1.6]{Bayrearrangement}, a sufficient condition is stated for a Dirichlet series $D$ to be rearrangement universality, that is for any $f\in H(\Omega)$,
	where $\Omega$ is the strip $\sigma_c(D)<\Re e(s)<\sigma_a(D),$
	there exists a permutation $\sigma$ of $\mathbb N$ such that $\sum_n a_{\sigma(n)} e^{-\lambda_{\sigma(n)}s}$
	converges to $f$ in $H(\Omega)$. This theorem is false. Indeed it would imply that $\sum_n (-1)^n n^{-s}$
	is rearrangement universal. This cannot hold: any rearrangement of $\sum_n (-1)^nn^{-s}$ will take values in $\mathbb R$
	for real values of the parameter $s$.
	The mistake that is made in \cite{Bayrearrangement} lies on the fact that a lemma due to Banaszczyk is only
	true for some real Fréchet spaces and was applied to the complex Fréchet space $H(\Omega)$.
	
	%%%%%%%%%%%%%%%%%%%%%%%%%%%%%%%%%%%
	%%%%%%%%%%%%%%%%%%%%%%%%%%%%%%%%%%%
	\section{Landau's Theorem revisited} \label{sec:landau}
	\subsection{Conditions (LC) and (WLC)}
	
	In this section, we investigate the existence of square moments for a general Dirichlet series $D(s)=\sum_{n\geq 1}a_n e^{-\lambda_n s}$. We shall need an assumption on the frequency $(\lambda_n)$ saying that two consecutive elements are not too close.
	
	\begin{definition}
		We say that a frequency $(\lambda_n)$ satisfies (LC) provided, for all $\delta>0$, there exists $C>0$ such that, for all $n\in\mathbb N,$
		$$\lambda_{n+1}-\lambda_n \geq Ce^{-e^{\delta\lambda_n}}.$$
	\end{definition}
	
	Many sequences satisfy (LC). For instance, if $P$ is a polynomial with $\lim_{+\infty}P=+\infty$, then the sequences $(\log(n))$ and $\log(P(p_n))$ satisfy (LC). Let us verify this for the latter sequence. For all $n\in\mathbb N,$
	\begin{align*}
	\log(P(p_{n+1}))-\log(P(p_n))&=\log\left(1+\frac{P(p_{n+1})-P(p_n)}{P(p_n)}\right)\\
	&\gtrsim \frac{P(p_{n+1})-P(p_n)}{P(p_n)}\\
	&\gtrsim \frac1{P(p_n)}\\
	&\gtrsim e^{-\log(P(p_n))}.
	\end{align*}
	
	In his book \cite{Landau09}, Landau has proved the following result (see Satz 37).
	
	\begin{theorem}\label{thm:landau}
		Let $(\lambda_n)$ be a frequency satisfying (LC). Let $D(s)=\sum_{n\geq 1}a_n e^{-\lambda_n s}$ be a Dirichlet series with finite abscissa of absolute convergence. Then for all $\sigma>\frac{\sigma_a(D)+\sigma_c(D)}2,$
		\begin{equation}\label{eq:Landau}
		\lim_{T\to+\infty}\frac1 T\int_0^T |D(\sigma+it)|^2dt =\sum_{n}|a_n|^2 e^{-2\lambda_n \sigma}.
		\end{equation}
	\end{theorem}
	Note that if $\limsup_n\frac{\log(n)}{\lambda_n}=\lambda$, then \eqref{eq:Landau} holds for every $\beta>\sigma_c(D)+\frac{\lambda}{2}$ since in that case $\sigma_a(D)\leq \sigma_c(D)+\lambda$.
	
	To verify that a Dirichlet series belongs to $\dwa(\sigma_0)$, one needs slightly more than Landau's theorem since we require that \eqref{eq:Landau} holds uniformly for all $\sigma\geq \beta$ with $\beta>\frac{\sigma_a(D)+\sigma_c(D)}2$. This stronger statement can be deduced from a careful inspection of the proof of \cite{Landau09}.
	
	The condition (LC) also appears in another very classical problem concerning Dirichlet series. Under which conditions do the partial sums converge uniformly to the associated Dirichlet series?
	In the classical setting, where $\lambda_n=\log(n)$, Bohr's theorem states:
	if a Dirichlet series converges somewhere and has a bounded analytic extension to $\{\Re e(s)>0\}$, then it converges uniformly in each half-plane $\{\Re(s)\geq \veps\}$, for all $\veps>0$. Recently, Bohr's theorem has been extended in an optimal way (see \cite{BK24}). We can derive the same conclusion with the weaker assumption that the range of the Dirichlet series omits two distinct points in the complex plane.
	
	For general sequences, the existence of a bounded analytic extension of a Dirichlet series does not always imply uniform convergence (see \cite{N22}). However, for Dirichlet series induced by a $\mathbb{Q}$-linear independent frequency, Bohr's theorem holds (see \cite{B13}). In the general setting, we need to require some separation between the terms of the frequency:
	if $(\lambda_n)$ satisfies (LC), then the analogue of Bohr's theorem holds (see \cite{L21}). In \cite{Baygendir}, it is shown that this last theorem can also be obtained with a relaxed condition on the sequence $(\lambda_n)$.
	
	\medskip
	
	In this section, we shall also extend Landau's theorem with a relaxed condition on the frequency.
	
	\begin{definition}
		We say that a frequency $\lambda$ satisfies (WLC) if for all $\delta>0$, there exists $C>0$ such that, for all $m\geq 2$, there exists $m_0,\,m_1\in\mathbb N$ with $m_0<m<m_1$ such that 
		\begin{align*}
		\lambda_{m_1}-\lambda_m &\geq Ce^{-e^{\delta\lambda_m}},\\
		\lambda_m-\lambda_{m_0}& \geq Ce^{-e^{\delta\lambda_{m_0}}},\\
		m_1-m_0&\leq Ce^{\delta\lambda_{m_0}}.
		\end{align*}
	\end{definition}
	
	Observe that any sequence satisfying (LC) also verifies (WLC): for all $m\geq 2$, one just takes $m_0=m-1$ and $m_1=m+1$. On the contrary, there are sequences satisfying (WLC) and not (LC). Indeed, define $(\lambda_n)$ by 
	$$\lambda_{2^n+k}=n^2+ke^{-e^{n^2}},\ k=0,\dots,2^n-1.$$
	Let $m=2^n+k,$ $n\geq 1$, let $\delta>0$ and define $m_0=2^{n-1}$, $m_1=2^{n+1}$. Then 
	\begin{align*}
	m_1-m_0&\leq 2^{n+1}\leq Ce^{\delta (n-1)^2}, \\
	\lambda_{m_1}-\lambda_{m}&\geq n\gtrsim e^{-e^{\delta n^2}}, \\
	\lambda_{m}-\lambda_{m_0}&\geq n\gtrsim e^{-e^{\delta (n-1)^2}}.
	\end{align*}
	Moreover, it is shown in \cite{Baygendir} that $(\lambda_n)$
	is not the finite union of frequencies satisfying (LC).
	
	\medskip
	
	We intend to prove the following precise version of Landau's theorem.
	
	\begin{theorem}\label{thm:ourlandau}
		Let $(\lambda_n)$ be a frequency satisfying (WLC). Let $D(s)=\sum_{n\geq 1}a_n e^{-\lambda_n s}$ be a Dirichlet series with finite abscissa of absolute convergence.
		Then for all $\beta> \frac{\sigma_a(D)+\sigma_c(D)}2$ and for all $\sigma\geq \beta,$
		\begin{equation}\label{eq:ourLandau}
		\lim_{T\to+\infty}\frac{1}{2T}\int_{-T}^T |D(\sigma+it)|^2dt =\sum_{n}|a_n|^2 e^{-2\lambda_n \sigma}
		\end{equation}
		uniformly in $\sigma\geq\beta.$
	\end{theorem}
	
	Since Landau's theorem seems to have been forgotten and since we need uniformity, we shall give a complete proof of Theorem \ref{thm:ourlandau} 
	in the next subsection. Of course, except for replacing (LC) by (WLC), we do not claim originality.
	
	\subsection{Proof of Landau's theorem}
	
	We start with a couple of lemmas.
	
	\begin{lemma}\label{lem:WLC}
		Let $\delta>0$ and $C>0$. Then there exists $C'>0$ such that, for any $a,b\in[1,+\infty)$:
		\begin{itemize}
			\item If $a-b\geq Ce^{-e^{b\delta}}$, then $\displaystyle \int_1^b \frac{e^{-u\delta}}{a-u}du\leq C'$.
			\item If $a-b\geq Ce^{-e^{a\delta}}$, then $\displaystyle \int_a^{+\infty}\frac{e^{-u\delta}}{u-b}du\leq C'$.
		\end{itemize}
	\end{lemma}
	\begin{proof}
		We just need to write
		\begin{align*}
		\int_1^b \frac{e^{-u\delta}}{a-u}du&\leq \int_1^{b-1}e^{-u\delta}du+\int_{b-1}^b \frac{e^{-u\delta}}{a-u}du\\
		&\lesssim 1+e^{-b\delta}\big(-\log(a-b)+\log(a-b+1)\big)\\
		&\lesssim 1+e^{-b\delta}\log\left(1+\frac 1{a-b}\right)\\
		&\lesssim 1.
		\end{align*}
		The proof of the other case is similar.
	\end{proof}
	
	In the following, we fix a frequency $(\lambda_n)$.
	
	\begin{lemma}\label{lem:uniformreste}
		Let $D(s)=\sum_{n=1}^{+\infty}a_n e^{-\lambda_n s}$ and $\sigma_0>\sigma_c(D).$ Then there exist $C>0$, $\delta>0$
		such that, for all $n\geq 1$ and for all $\sigma\geq\sigma_0,$
		$$\left|\sum_{k=n}^{+\infty}a_k e^{-\lambda_k\sigma}\right|\leq Ce^{-\lambda_n\delta}.$$
	\end{lemma}
	
	This is \cite[Satz 34]{Landau09}. For expository reasons, we will present an elementary proof of it.
	\begin{proof}
		The Dirichlet series converges at $\frac{\sigma_c(D)+\sigma_0}{2}$ and as a consequence the partial sums 
		$$A(n)=\sum_{k=1}^{n}a_k e^{-\lambda_k\frac{\sigma_c(D)+\sigma_0}{2}}$$
		are bounded by a constant $C'>0$. A summation by parts yields to
		\[\left|\sum_{k=n}^{+\infty}a_k e^{-\lambda_k\sigma}\right|\leq 2C'e^{-\lambda_n\frac{2\sigma-\sigma_0-\sigma_c(D)}{2}}.\]
		Thus, we can choose
		\begin{equation*}
		\delta=\frac{\sigma_0-\sigma_c(D)}2,\qquad
		C=2\sup_{n\geq 1}\left|A(n)\right|.
		\end{equation*}  
	\end{proof}
	
	The next lemma is the key point and the only place where we use (WLC).
	\begin{lemma}\label{lem:kwlc}
		Assume that $(\lambda_n)$ satisfies (WLC) and let $D(s)=\sum_{n=1}^{+\infty}a_n e^{-\lambda_n s}$. For every $\sigma_0>\sigma_c(D),$ there exist $C>0$ and $\delta>0$ such that, for all $r\in\NN,$
		for all $m\in\NN,$ for all $T\geq 1,$ for all $\sigma\geq \sigma_0,$
		$$\left|\frac 1{2T}\int_{-T}^T e^{\lambda_m it}\sum_{\substack{n=r\\n\neq m}}^{+\infty}
		a_n e^{-\lambda_n(\sigma+it)}dt\right|\leq Ce^{-\lambda_r \delta }.$$
	\end{lemma}
	\begin{proof}
		In what follows, the implied constant in the notation $f(m,r,T,\sigma)\lesssim g(m,r,T,\sigma)$ will never depend on $m\in\mathbb N$, $r\in\mathbb N$, $T\geq 1$ and $\sigma\geq\sigma_0$. For $n\in\mathbb N$ and $\sigma\geq\sigma_0$, we denote by $R_n(\sigma)$ the quantity
		$$R_n(\sigma)=\sum_{k=n}^{+\infty}a_k e^{-\lambda_k \sigma}.$$
		By Lemma \ref{lem:uniformreste}, there exists $\delta>0$ such that, for all $n\in\mathbb N$ and all $\sigma\geq\sigma_0,$
		\begin{equation}\label{eq:reste}
		|R_n(\sigma)|\lesssim e^{-\lambda_n\delta}.
		\end{equation}
		We use that $(\lambda_n)$ satisfies (WLC) for $\delta/2$ to get, for each $m\geq 2,$ the existence of $m_0<m<m_1$ satisfying
		\begin{align*}
		\lambda_{m_1}-\lambda_m &\gtrsim e^{-e^{\delta\lambda_m/2}},\\
		\lambda_m-\lambda_{m_0}& \gtrsim e^{-e^{\delta\lambda_{m_0}/2}},\\
		m_1-m_0&\lesssim e^{\delta\lambda_{m_0}/2}.
		\end{align*}
		We fix $m\geq 2$ and $r\in\mathbb N$. We first assume that $r\leq m_0$. Then we split the sum $\sum_{n=r}^{+\infty}$ into three parts: $\sum_{n=r}^{m_0-1}$, $\sum_{m_0}^{m_1-1}$ and $\sum_{m_1}^{+\infty}$. We first observe that, for all $\sigma\geq \sigma_0$ and all $T\geq 1,$
		\begin{align*}
		\left|\frac 1{2T}\int_{-T}^T e^{\lambda_m it}\sum_{\substack{n=m_0\\n\neq m}}^{m_1-1}
		a_n e^{-\lambda_n(\sigma+it)}dt\right|&\leq (m_1-m_0)\sup_{n\in[m_0,m_1-1]}|a_n|e^{-\lambda_n  \sigma}\\
		&\lesssim e^{\delta\lambda_{m_0}/2}e^{-\delta\lambda_{m_0}}\\
		&\lesssim e^{-\delta\lambda_{m_0}/2}.
		\end{align*}
		By a summation by parts, for all $\sigma\geq\sigma_0$ and all $t\in\mathbb R$,
		\begin{align*}
		\sum_{n=m_1}^{+\infty}a_n e^{-\lambda_n(\sigma+it)} &=\sum_{n=m_1+1}^{+\infty} R_n(\sigma)
		\big(e^{-\lambda_n it}-e^{-\lambda_{n-1}it}\big)+R_{m_1}(\sigma)e^{-\lambda_{m_1}it}\\
		&=-\sum_{n=m_1+1}^{+\infty}R_n(\sigma)it\int_{\lambda_{n-1}}^{\lambda_n}e^{-iut}du+R_{m_1}(\sigma)e^{-\lambda_{m_1}it}.
		\end{align*}
		We multiply by $e^{\lambda_m it}$ and integrate over $[-T,T],\,T\geq 1$:
		\begin{align*}
		\int_{-T}^T e^{\lambda_m it}\sum_{n=m_1}^{+\infty}a_n e^{-\lambda_n(\sigma+it)}dt&=-\sum_{n=m_1+1}^{+\infty}R_n(\sigma)\int_{\lambda_{n-1}}^{\lambda_n}du\int_{-T}^T ite^{it(\lambda_m-u)}dt \\
		&\quad+R_{m_1}(\sigma)\int_{-T}^T e^{(\lambda_m-\lambda_{m_1})it}dt.
		\end{align*}
		By an integration by parts and the estimate \eqref{eq:reste}, 
		\begin{align*}
		\left| \sum_{n=m_1+1}^{+\infty}R_n(\sigma)\int_{\lambda_{n-1}}^{\lambda_n}du\int_{-T}^T ite^{it(\lambda_m-u)}dt\right|&\lesssim \sum_{n=m_1+1}^{+\infty} e^{-\lambda_n\delta}\int_{\lambda_{n-1}}^{\lambda_n} \frac T{u-\lambda_m}du\\
		&\lesssim T\sum_{n=m_1+1}^{+\infty}e^{-\lambda_n\delta/2}\int_{\lambda_{n-1}}^{\lambda_n}\frac{e^{-u\delta/2}}{u-\lambda_m}du\\
		&\lesssim T e^{-\lambda_r\delta/2} \int_{\lambda_{m_1}}^{+\infty}\frac{e^{-u\delta/2}}{u-\lambda_m}du\\
		&\lesssim Te^{-\lambda_r\delta/2},   
		\end{align*}
		in the last step we applied Lemma \ref{lem:WLC}. Moreover
		$$\left |R_{m_1}(\sigma)\int_{-T}^T e^{(\lambda_m-\lambda_{m_1})it}dt\right|\lesssim Te^{-\lambda_{m_1}\delta}\leq Te^{-\lambda_r \delta}.$$
		To estimate the sum $\sum_{n=r}^{m_0}$ we work in a similar manner as above 
		using the first half of Lemma \ref{lem:WLC}. If $r$ is not less than $m_0,$ we follow the same strategy but there are less terms to consider.
	\end{proof}
	
	Now, we can approximate uniformly the coefficients of a Dirichlet series.
	
	\begin{theorem}\label{thm:landaulike}
		Assume that $(\lambda_n)$ satisfies $(WLC)$ 
		and let $D(s)=\sum_{n\geq 1}a_n e^{-\lambda_n s}$ be a Dirichlet series. Then for every
		$\sigma_0>\sigma_c(D)$ and $\veps>0$ there exists $T_0\geq 1$ such that, for all $T\geq T_0$, for all $m\geq 1$, for all $\sigma\geq\sigma_0$,
		$$\left|\frac 1{2T}\int_{-T}^T e^{\lambda_m it}D(\sigma+it)dt-a_me^{-\lambda_m \sigma}\right|\leq\veps.$$
	\end{theorem}
	\begin{proof}
		By Lemma \ref{lem:kwlc}, there exists $r\geq 1$ such that, for all $T\geq 1$, for all $m\geq 1$ and all $\sigma\geq \sigma_0,$
		$$\left|\frac 1{2T}\int_{-T}^T e^{\lambda_m it}\sum_{\substack{n=r\\n\neq m}}^{+\infty}
		a_n e^{-\lambda_n(\sigma+it)}dt\right|\leq\veps.$$
		Writing 
		$$D(\sigma+it)=\sum_{\substack{n=1\\n\neq m}}^{r-1}a_n e^{-\lambda_n (\sigma+it)}+\sum_{\substack{n=r\\n\neq m}}^{+\infty}a_ne^{-\lambda_n(\sigma+it)}+a_me^{-\lambda_m(\sigma+it)},$$
		we just need to prove that there exists $T_0\geq 1$ such that, for all $T\geq T_0$,
		for all $m\geq 1,$ for all $\sigma\geq\sigma_0$
		$$\left|\frac 1{2T}\int_{-T}^T e^{\lambda_m it}\sum_{\substack{n=1\\n\neq m}}^{r-1}
		a_n e^{-\lambda_n(\sigma+it)}dt\right|\leq\veps.$$
		Now,
		\begin{align*}
		\left|\frac 1{2T}\int_{-T}^T e^{\lambda_m it}\sum_{\substack{n=1\\n\neq m}}^{r-1}
		a_n e^{-\lambda_n(\sigma+it)}dt\right|&\leq \frac 1{2T}\sum_{\substack{n=1\\n\neq m}}^{r-1}|a_n|e^{-\lambda_n \sigma}\left|\int_{-T}^T e^{i(\lambda_m-\lambda_n)t}dt\right|\\
		&\leq \frac 1T \sum_{n=1}^{r-1}|a_n|e^{-\lambda_n\sigma_0}\sup_{j=1,\dots,r-1}\frac1{\lambda_{j+1}-\lambda_j}
		\end{align*}
		which yields the result. 
	\end{proof}
	Schnee proved a non-uniform version of the above theorem with the extra assumption of (LC). In the book of Landau \cite{Landau09}, the proof of Schnee's theorem gives the same result without the extra assumption of (LC), see also \cite{Helson63}.
	
	Absolute convergence and Theorem~\ref{thm:landaulike} yield to the following corollary:
	\begin{corollary}\label{absmoment}
		Assume that $(\lambda_n)$ satisfies $(WLC)$ and that $D(s)=\sum_{n\geq 1}a_n e^{-\lambda_n s}$, $f(s)=\sum_{n\geq 1}b_n e^{-\lambda_n s}$ are Dirichlet series with finite abscissas of convergence and of absolute convergence, respectively. 
		Let $\sigma_0>\sigma_c(D)$, $\sigma_1>\sigma_a(f)$  and $\veps>0$. Then there exists $T_0\geq 1$ such that, for all $T\geq T_0$, for all $\sigma\geq\sigma_0$ and $x\geq\sigma_1$,
		$$\left|\frac 1{2T}\int_{-T}^{T} D(\sigma+it)\overline{f(x+it)}dt-\sum_{n\geq 1}a_n\overline{b_n} e^{-\lambda_n(\sigma+x)}\right|\leq\veps.$$
	\end{corollary}
	
	The last lemma that we need for the proof of Landau's theorem is a well-known estimate for the order of a convergent Dirichlet series, see \cite{HarRi}.
	\begin{lemma}
		Let $D(s)=\sum_{n\geq 1}a_n e^{-\lambda_n s}$ be a Dirichlet series with finite abscissa of absolute convergence and assume that $\tau=\sigma_a(D)-\sigma_c(D)>0$. Then for every $\veps>0$, there exists $C(\veps)>0$ such that, for all $\sigma\in(\sigma_c(D)+\varepsilon,\sigma_a(D)]$, 
		\begin{equation}\label{eq:order}
		|D(\sigma+it)|\leq C |t|^{1-\frac{\sigma-\sigma_c(D)-\frac{\varepsilon}{2}}{\tau}},\qquad |t|\geq 1.
		\end{equation}
	\end{lemma}
	For expository reasons, we present the classical argument.
	\begin{proof}
		The Dirichlet series converges at $\sigma_c(D)+\frac{\varepsilon}{2}$. Thus, the partial sums 
		$$A(n):=\sum_{k=1}^{n}a_k e^{-\lambda_k(\sigma_c(D)+\frac{\varepsilon}{2})}$$
		are bounded by a constant $M>0$. A summation by parts yields 
		\begin{align*}
		|D(\sigma+it)|&\lesssim \sum_{n=1}^{N}|a_n| e^{-\lambda_n\sigma}+|t|\sum_{n>N}\int_{\lambda_n}^{\lambda_{n+1}}e^{-u(\sigma-\sigma_c(D)-\frac{\varepsilon}{2})}\,du+ e^{-\lambda_{N+1}(\sigma-\sigma_c(D)-\frac{\varepsilon}{2})}\\
		&\lesssim e^{\lambda_N(\sigma_a(D)+\frac{\varepsilon}{2}-\sigma)}+|t|e^{-\lambda_{N+1}(\sigma-\sigma_c(D)-\frac{\varepsilon}{2})}.
		\end{align*}
		The desired estimate follows if we choose $\lambda_N$ to be the largest term of the frequency that is not greater than $\frac{\log|t|}{\tau}$.
	\end{proof}
	\begin{proof}[Proof of Landau's theorem]
		We observe that
		\[\int_{-T}^T|D(\sigma+it)|^2\,dt=-i\int_{-iT+\sigma}^{iT+\sigma}D(s)\widetilde{D}(2\sigma-s)\,ds,\]
		where $\widetilde{D}(s)=\overline{D(\overline{s})}$ is the symmetric holomorphic function of $D$ with respect to the real line. We apply Cauchy's theorem, yielding to:
		\begin{multline*}
		\int_{-T}^T|D(\sigma+it)|^2\,dt-\int_{-T}^T D(\sigma_c(D)+\varepsilon+it)\overline{D(2\sigma-\sigma_c(D)-\varepsilon+it)}\,dt\\
		=\frac{1}{i}\int_{\sigma_c(D)+\varepsilon-iT}^{\sigma-iT}D(s)\widetilde{D}(2\sigma-s)\,ds-\frac{1}{i}\int_{\sigma_c(D)+\varepsilon+iT}^{\sigma+iT}D(s)\widetilde{D}(2\sigma-s)\,ds,
		\end{multline*}
		for $\varepsilon>0$ sufficiently small. By \eqref{eq:order} the integrals in the right hand side are $O(T^A)$, where $A=A(\varepsilon)<1$ and all the associated constants depend only on $\varepsilon>0$ and on $D$. Choosing $\veps>0$ such that $2\beta-\sigma_c(D)-\veps>\sigma_a(D)$, this and Corollary~\ref{absmoment} imply that 
		\[\lim_{T\rightarrow+\infty}\frac{1}{2T}\int_{-T}^T|D(\sigma+it)|^2\,dt=\sum_{n\geq 1}|a_n|^2e^{-2\sigma\lambda_n},\]
		and the convergence is uniform for $\sigma\geq\beta>\frac{\sigma_a(D)+\sigma_c(D)}{2}$.
	\end{proof}
	\subsection{Application to the square moments of the zeta function}
	
	Landau's theorem gives a direct and elementary proof of the fact that the square moments of the zeta function are finite. At least in the recent literature, all the proofs go through the method of non-stationary phase and Euler's summation formula, see Section~\ref{5}. 
	\begin{theorem}
		For $\sigma\in\left(\frac{1}{2},1\right)$
		\begin{equation}
		\lim_{T\to+\infty}\frac1 T\int_0^T |\zeta(\sigma+it)|^2dt =\zeta(2\sigma),
		\end{equation}
		uniformly in $\sigma\in (\sigma_1,\sigma_2)$, where $\frac{1}{2}<\sigma_1<\sigma_2<1$.
	\end{theorem}
	\begin{proof}
		We consider the Dirichlet eta function $\eta(s)=\sum\limits_{n\geq1}(-1)^{n-1}n^{-s},\, \Re e(s)>0$. It is easy to verify that $\sigma_c(\eta)=0$, $\sigma_a(\eta)=1$ and that
		\begin{equation*}
		\eta(s)=(1-2^{1-s})\zeta(s), \qquad \Re e(s)>0,\, s\neq1.
		\end{equation*}
		Applying Landau's theorem to the eta function, we obtain
		\[\int_0^T |\zeta(\sigma+it)|^2dt =O(T),\]
		uniformly in $\sigma\in(\frac{1}{2}+\varepsilon,1-\varepsilon)$.
		To prove that the limit is $\zeta(2\sigma)$, one needs to repeat the same argument for the functions
		\[\eta_N(s)=\sum_{n\geq1}a_{n,N}n^{-s}=(1-2^{1-s})(\zeta(s)-\zeta_N(s))=\sum_{n>2N}(-1)^{n-1}n^{-s}+\sum_{n=N+1}^{2N}n^{-s},\]
		where $\zeta_N(s)=\sum_{n=1}^Nn^{-s}$ are the partial sums of the zeta function. 
		
		For every $\varepsilon>0$, there exists $N>0$ and $T_0(\veps,\, N,\, \sigma_0,\,  \sigma_1)\geq 1$, such that for every $T\geq T_0$
		\begin{align*}
		\left|\frac{1}{T}\int_0^T|\zeta(\sigma+it)|^2\,dt-\sum_{n\geq 1}n^{-2\sigma}\right|^\frac{1}{2}&\leq \left(\frac{1}{T}\int_0^T|\zeta(\sigma+it)-\zeta_N(\sigma+it)|^2\,dt\right)^{\frac{1}{2}}+\varepsilon\\
		&\lesssim\left(\frac{1}{T}\int_0^T|\eta_N(\sigma+it)|^2\,dt\right)^{\frac{1}{2}}+\varepsilon\\
		&\leq 3\varepsilon.\qedhere
		\end{align*}
	\end{proof}
	
	\section{Examples of strongly universal convergent Dirichlet series }\label{sec:examples}
	
	We shall use Theorem~\ref{thm:ourlandau} to give a large class of general Dirichlet series belonging to some $\dwa(\sigma_0)$. We first prove Theorem \ref{thm:main}.
	\begin{proof}[Proof of Theorem \ref{thm:main}]
		Landau's theorem implies that $D\in \dwa((\sigma_c(D)+\sigma_a(D))/2)$ and by assumption $D\in \ddens$. Hence the result follows from Theorem \ref{thm:bayrearrangement}.
	\end{proof}
	
	If we know more precisely the growth of both $(a_n)$ and $(\lambda_n)$, one can replace \eqref{eq:thmmain} by a condition involving only the frequency $(\lambda_n)$.

	\begin{corollary} \label{cor:main}
		Let $(\lambda_n)$ be a frequency, let $D(s)=\sum_{n\geq 1}a_n e^{-\lambda_n s}$ be a Dirichlet series and let $d>0$. Assume that 
		\begin{itemize}
			\item[(a)]  The sequence $(a_n)$ satisfies 
			$$\lim_{n\to+\infty}\frac{\log(|a_n|)}{\log(n)}=d-1.$$
			\item[(b)] The frequency $(\lambda_n)$ is $\mathbb Q$-linearly independent, satisfies (WLC) and 
			$$\lim_{n\to+\infty}\frac{\lambda_n}{\log(n)}=d.$$
			\item[(c)] For all $\alpha,\beta>0,$ there exist $C>0$ and $x_0\geq 1$ such that, for all $x\geq x_0,$ 
			$$\textrm{card}\left(\left\{n\in\mathbb N:\ \lambda_n \in\left[x,x+\frac{\alpha}{x^2}\right]\right\}\right)\geq Ce^{x\left(\frac 1d-\beta\right)}.$$\label{ccc}
		\end{itemize}
		Then $D$ is strongly universal in $\{(\sigma_c(D)+1)/2<\Re e(s)<1\}$. 
	\end{corollary}
	\begin{proof}
		The assumptions easily imply that $\sigma_a(D)=1$. It remains to verify   
		that condition \eqref{eq:thmmain} is satisfied. 
		Let $\alpha,\beta>0$ and let $\veps>0$ be very small. Then, provided $\lambda_n\in[x,x+\alpha/x^2]$ and $x$ is large enough,
		$$n\geq e^{\frac x{d+\veps}}$$
		so that 
		$$|a_n|\geq e^{\frac{d-1-\veps}{d+\veps}x}.$$
		Therefore,
		$$\sum_{\lambda_n \in \left[x,x+\frac{\alpha}{x^2}\right]}|a_n|\geq \exp\left(x\left(1-\frac{1+2\veps}{d+\veps}\right)\right)\textrm{card}\left(\left\{n\in\mathbb N:\ \lambda_n \in\left[x,x+\frac{\alpha}{x^2}\right]\right\}\right).$$
		We use condition (c) with $\alpha$ and with $\beta_0>0$ very small and we get 
		\begin{align*}
		\sum_{\lambda_n \in \left[x,x+\frac{\alpha}{x^2}\right]}|a_n|&\geq C \exp\left(x\left(1-\left(\frac{1+2\veps}{d+\veps}-\frac 1d+\beta_0\right)\right)\right)\\
		&\geq C e^{x(1-\beta)}
		\end{align*}
		provided $\beta_0>0$ and $\veps>0$ are small enough.
	\end{proof}
	
	We now provide examples of frequencies $(\lambda_n)$ satisfying Condition (c) of Corollary \ref{cor:main}. But first, let us state the following lemma proved in \cite[Lemma~6.1]{Bayrearrangement}:
	\begin{lemma}\label{lemma61}
		Let $P(X)=\sum_{k=0}^{d}b_kX^k$ be a polynomial of degree $d$, with $b_d>0$. Then, there exist $x_0,\,y_0>0$ such that $P$ induces a bijection from $[x_0,+\infty]$ to $[y_0,+\infty]$, and
		\[P^{-1}(x)=\frac{x^{1/d}}{(b_d^{1/d})}-\frac{b_{d-1}}{b_d^{(d-1)/d}}+o(1),\]
		as $x\rightarrow+\infty$.
	\end{lemma}

	\begin{example}\label{ex:ddens}
		Let $d\geq 1,$
		let $P\in\mathbb R[X]$ with $\deg(P)=d,$ $\lim_{+\infty}P=+\infty$ and let $\lambda_n=\log(P(n))$ for $n\geq n_0,$ where $P\geq 0$ on $[n_0,+\infty)$. Then for all $\alpha,\beta>0,$ there exist $C>0$ and $x_0\geq 1$ such that, for all $x\geq x_0,$ 
		$$\textrm{card}\left(\left\{n\in\mathbb N:\ \lambda_n \in\left[x,x+\frac{\alpha}{x^2}\right]\right\}\right)\geq Ce^{x\left(\frac 1d-\beta\right)}.$$
	\end{example}
	\begin{proof}
		Let $\alpha,\beta>0$. Without loss of generality, we may assume that $P$ is one-to-one on $[n_0,+\infty)$. Then 
		\[\lambda_n\in\left[x,x+\frac\alpha{x^2}\right]\qquad\text{if and only if}\qquad n\in \left[P^{-1}(e^x),P^{-1}(e^{x+\frac\alpha{x^2}})\right].\]
		Using Lemma~\ref{lemma61}, there exist $c_1>0,\,c_2\in\mathbb{R}$ such that for small $\varepsilon>0$ and for every $x$ sufficiently large:
		\[ n\in \left[c_1e^{\frac xd}+c_2+\varepsilon,c_1 e^{\frac xd+\frac{\alpha}{dx^2}}+c_2-\varepsilon\right]\qquad\text{implies that}\qquad\lambda_n\in \left[x,x+\frac{\alpha}{x^2}\right].\]
		This yields the result since 
		$$\textrm{card}\left(\left\{n:\ \lambda_n\in\left[x,x+\frac{\alpha}{x^2}\right]\right\}\right)\gtrsim e^{\frac xd}\left(e^{\frac \alpha{x^2}}-1\right)\gtrsim \frac{e^{\frac xd}}{x^2}\gtrsim e^{x(\frac{1}{d}-\beta)}.$$
	\end{proof}
	
	Corollary~\ref{exten} now follows from Corollary \ref{cor:main} and Example \ref{ex:ddens}.
	Moreover, if $(a_n)$ is a sequence of positive real numbers such that, for all $\veps>0$,
	the sequence $(a_n (P(n))^{-(1+\veps-1/d)})$ is eventually nonincreasing, a summation by parts yields $\sigma_c(D)=1-1/d$
	where $D(s)=\sum_{n}e^{i\omega n}a_n (P(n))^{-s},$ $\omega\in\mathbb R\backslash 2\pi\mathbb Z$. We therefore get as a particular case Theorem A.
	
	This also allows us to give a new proof of the universality of the alternating Hurwitz zeta function.
	\begin{corollary}
		The Dirichlet series $\sum_{n\geq 1}e^{i\omega n} (n+\alpha)^{-s},$ where $\omega\in\mathbb{R}\setminus2\pi\mathbb{Z}$ and $\alpha$ is transcendental, is strongly universal in $\{1/2<\Re e(s)<1\}$. 
	\end{corollary}

	We now handle the case of frequencies induced by the sequence of prime numbers.
	\begin{example}\label{ex:ddensprime}
		Let $d\geq 1,$ let $P\in\mathbb R_d[X]$ with $\deg(P)=d$ and $\lim_{+\infty}P=+\infty$ and let $\lambda_n=\log(P(p_n))$ for $n\geq n_0,$ where $P\geq 0$ on $[p_{n_0},+\infty)$. Then for all $\alpha,\beta>0,$ there exist $C>0$ and $x_0\geq 1$ such that, for all $x\geq x_0,$ 
		$$\textrm{card}\left(\left\{n\in\mathbb N:\ \lambda_n \in\left[x,x+\frac{\alpha}{x^2}\right]\right\}\right)\geq Ce^{x\left(\frac 1d-\beta\right)}.$$
	\end{example}
	\begin{proof}
		Arguing as above, we know that there exist $c_1>0,$ $c_2\in\mathbb R$ such that, 
		for small $\varepsilon>0$ and for every $x$ sufficiently large:
		\[ p_n\in \left[c_1e^{\frac xd}+c_2+\varepsilon,c_1 e^{\frac xd+\frac{\alpha}{dx^2}}+c_2-\varepsilon\right]\qquad\text{implies that}\qquad\lambda_n\in \left[x,x+\frac{\alpha}{x^2}\right].\]
		By Hadamard - De la Vallée Poussin's estimate, we also know that 
		\begin{align*}
		\Pi(u)&:=\textrm{card}(\{n:\ p_n\leq u\})\\
		&=\int_2^u \frac{dt}{\log(t)}+O\left(ue^{-c\sqrt{\log u}}\right).
		\end{align*}
		Therefore
		\begin{align*}
		\textrm{card}\left(\left\{n:\ p_n\in \left[c_1e^{\frac xd}+c_2+\varepsilon,c_1 e^{\frac xd+\frac{\alpha}{dx^2}}+c_2-\varepsilon\right]\right\}\right)
		&\gtrsim \int_{c_1e^{\frac xd}-c_2+\varepsilon}^{c_1 e^{\frac xd+\frac{\alpha}{dx^2}}-c_2-\varepsilon} \frac{dt}{\log(t)}+O\left(e^{\frac xd-c\sqrt{x}}\right)\\
		&\gtrsim e^{\frac xd}\frac{e^{\frac{\alpha}{dx^2}}-1}{x}+O\left(e^{\frac xd-c'\sqrt{x}}\right)\\
		&\gtrsim \frac{e^{\frac{x}{d}}}{x^{3}}\\
		&\gtrsim e^{(\frac{1}{d}-\beta)x}.
		\end{align*}		
	\end{proof}
	
	Again Corollary~\ref{exprime} follows immediately from Corollary~\ref{cor:main} and Example~\ref{ex:ddensprime}.
	
	\begin{corollary}
		The Dirichlet series $\sum_{n\geq 1}e^{i\omega n} p_n^{-s},\,\omega\in\mathbb{R}\setminus2\pi\mathbb{Z}$, is strongly universal in $\{1/2<\Re e(s)<1\}$. 
	\end{corollary}
	
	Thus, the alternating prime zeta function is strongly universal on the critical strip; this gives a positive answer to a question posed by the first author in \cite{Bayrearrangement}.
	
	\section{Proof of Theorem \ref{thm:main2}}\label{5}
	We have to face a new difficulty since $D$ will now be defined via an analytic continuation. We need to understand how to define this analytic continuation and how close it is to the partial sums of $D$ in order to be able to show that D satisfies conditions (c) and (d) of $\dwa$.
	
	\begin{lemma}\label{lem:analyticcontinuation}
		Let $d\geq 1,$ let $P\in\mathbb R[X]$ with $\deg(P)=d$ and $\lim_{+\infty}P=+\infty,$ let $Q\in\mathbb R[X]$ with $\deg(Q)=d-1$ and let $\kappa\in\mathbb R$.
		Then the Dirichlet series $D(s)=\sum_n Q(n)(\log n)^\kappa (P(n))^{-s}$ admits a holomorphic continuation
		to $\CC^+_{1-\frac 1d}\cup\CC_1$ and even to $\CC_{1-\frac 1d}\backslash\{1\}$ provided $\kappa\in\mathbb N_0.$ Moreover, let $\sigma_1>1-\frac 1d$ and $\sigma_2>1.$
		\begin{enumerate}[(a)]
			\item There exist $t_0,\ B>0$ such that, for all $s=\sigma+it$ with $\sigma\geq \sigma_1$ and $t\geq t_0,$
			$$|D(s)|\leq t^B.$$
			\item There exist $\delta,\,\veps>0$ such that, for all $x>0,$ for all $s=\sigma+it$ with $\sigma\in[\sigma_1,\sigma_2]$
			and $1\leq t\leq  \delta x,$ 
			$$D(s)=\sum_{n=2}^x Q(n)(\log n)^\kappa (P(n))^{-s}+O(x^{-\veps})+O\left(\frac{(\log P(x))^\kappa}{(s-1)P(x)^{s-1}}\right)$$
			(here, the $O$-constants do not depend neither on $s$ nor on $x$).
		\end{enumerate}
	\end{lemma}
	\begin{proof}
		As in the classical case of the Riemann zeta function, see for example \cite{BM09}, our plan is to use the regularity and smoothness of the coefficients and the frequencies of our Dirichlet series $D$ to estimate its order and how close the partial sums approximate $D$. We will rely again on the principle of non-stationary phase, that is Lemma~\ref{lem:vandercorput}. But first we need to deal with some technical difficulties that arise from the "unknown" polynomials $P$ and $Q$. 
		We start with $s=\sigma+it$, $\sigma>1$ and let $N\geq 1.$ We write
		$$D(s)=\sum_{n=2}^{N-1}Q(n)(\log n)^\kappa (P(n))^{-s}+\sum_{n=N}^{+\infty}Q(n)(\log n)^\kappa (P(n))^{-s}$$
		and we apply Euler's summation formula (see \cite[(11.3)]{BM09}). Setting 
		$$\phi(u)=Q(u)(\log u)^\kappa (P(u))^{-s}\textrm{ and }\rho(u)=u-\lfloor u\rfloor-\frac 12,$$
		we get 
		\begin{equation}\label{eq:analytic1}
		D(s)=\sum_{n=2}^{N-1}Q(n)(\log n)^\kappa (P(n))^{-s}+\int_N^{+\infty}\phi(u)du+\int_{N}^{+\infty}\rho(u)\phi'(u)du+\frac 12\phi(N).
		\end{equation}
		These integrals are convergent when $s\in\CC_1.$ Moreover it is easy to check that there exists $\veps>0$ such that, provided $s=\sigma+it$ with $\sigma\geq \sigma_1>1-\frac{1}{d},$ for any $u>2,$
		$$|\phi(u)|\lesssim u^{-\veps}\textrm{ and }|\phi'(u)|\lesssim |s| u^{-1-\veps}.$$
		In particular, the last integral in \eqref{eq:analytic1} defines a holomorphic function in $\CC_{1-\frac 1d}$
		which is $O(|s|N^{-\veps})$ in $\mathbb C_{\sigma_1}$.  Let us now see how to control the first integral. Up to multiply $Q$ by some constant, we may
		write it $Q(u)=P'(u)+Q_1(u)$ with $\deg(Q_1)\leq d-2.$
		As above, the integral $\int_N^{+\infty}Q_1(u)(\log u)^\kappa (P(u))^{-s}du$ 
		defines an analytic function in $\CC_{1-\frac 1d}$ which is $O(N^{-\veps})$. Therefore we have obtained so far that $D$ may be written in $\CC_1$
		$$D(s)=\sum_{n=2}^{N-1}Q(n)(\log n)^\kappa (P(n))^{-s}+\int_N^{+\infty}\frac{P'(u)(\log u)^\kappa}{(P(u))^s}du+R_N(s)$$
		where $R_N$ is analytic in $\CC_{1-\frac 1d}$ and $|R_N(s)|\lesssim |s|N^{-\veps}$ uniformly for $\sigma\geq\sigma_1.$
		
		We choose $N$ sufficiently large such that $P$ is one-to-one on $[N,+\infty)$. By change of variables we obtain:
		$$\int_N^{+\infty}\frac{P'(u)(\log u)^\kappa}{(P(u))^s}du=\int_{P(N)}^{+\infty}\frac{(\log P^{-1}(u))^\kappa}{u^s}du.$$
		
		By Lemma~\ref{lemma61} we have the following formula:
		\[P^{-1}(u)=a_d u^{1/d}(1+\veps_1(u))\textrm{ with }|\veps_1(u)|\lesssim u^{-1/d},\]
		where $a_d>0$.
		Therefore,
		$$(\log P^{-1}(u))^\kappa=\log^\kappa(a_d u^{1/d})+\veps_2(u)$$
		with 
		$$|\veps_2(u)|\lesssim u^{-1/d}\log^{\kappa-1}(u).$$
		As before, the integral $\int_{P(N)}^{+\infty}\veps_2(u)u^{-s}du$ defines an analytic function in the half-plane $\mathbb C_{1-\frac 1d}$ which is $O(P(N)^{-\veps})$ in $\mathbb C_{\sigma_1}$.
		On the other hand, setting $b_d=a_d^d$ and restricting ourselves to $s\in\mathbb{C}_1,$ we may write
		\begin{align*}
		\int_{P(N)}^{+\infty} \frac{\log^\kappa(a_d u^{1/d})}{u^s}du&=
		\int_{P(N)}^{+\infty} \frac 1{d^\kappa} \frac{\log^\kappa(b_d u)}{u^s} du\\
		&=\frac{b_d^{s-1}}{d^\kappa}\int_{b_d P(N)}^{+\infty}\frac{\log^\kappa(v)}{v^s}dv&\quad (v=b_du)\\
		&=\frac{b_d^{s-1}}{d^\kappa}\int_{\log(b_d P(N))}^{+\infty} 
		y^\kappa e^{(1-s)y}dy&\quad (y=\log v)\\
		&=\frac{b_d^{s-1}}{d^\kappa(s-1)^{\kappa+1}} \Gamma(\kappa+1,(s-1)\log(b_d P(N))).
		\end{align*}
		Hence we have shown that for $s\in\CC_1,$
		we may write
		\begin{align*}
		D(s)&=\sum_{n=2}^{N-1}Q(n)(\log n)^\kappa (P(n))^{-s}+\widetilde{R_N}(s)+\\
		&\quad\quad \frac{b_d^{s-1}}{d^\kappa(s-1)^{\kappa+1}} \Gamma(\kappa+1,(s-1)\log(b_d P(N)))
		\end{align*}
		where $\widetilde{R_N}(s)$ is holomorphic in $\CC_{1-\frac 1d}$ and is $O(|s|N^{-\veps})+O(P(N)^{-\veps})$ in $\mathbb C_{\sigma_1}.$
		Since we know that $\Gamma(\kappa+1,\cdot)$ admits an analytic continuation to $\mathbb C\backslash\mathbb R_-$
		we can conclude to the analytic continuation of $D$ to $\CC_{1-\frac 1d}^+\cup\CC_1$. When $\kappa\in\mathbb N,$
		the analytic continuation even holds on $\CC_{1-\frac 1d}\backslash\{1\}.$
		The estimation (a) (which is trivial for $\sigma\geq \sigma_2>1$) follows easily for $\sigma\in[\sigma_1,\sigma_2]$
		by what we already know on $\widetilde{R_N}$ and by \eqref{eq:ineggamma}.

		Let us turn to the proof of (b).
		Choosing $N\geq x,$ we may write
		\begin{align*}
		D(s)&=\sum_{n=2}^{x}Q(n)(\log n)^\kappa (P(n))^{-s}+\sum_{n=x+1}^{N}Q(n)(\log n)^\kappa (P(n))^{-s}+\\
		&\quad\quad \frac{b_d^{s-1}}{d^\kappa(s-1)^{\kappa+1}} \Gamma(\kappa+1,(s-1)\log(b_d P(N)))+O(|s|N^{-\veps})+O(P(N)^{-\veps}).
		\end{align*}
		We apply Lemma \ref{lem:vandercorput} to the second sum with 
		$$g(u)=Q(u)\log^\kappa(u)(P(u))^{-\sigma},\ f(u)=\frac{-t\log(P(u))}{2\pi}.$$
		Observe that, for $\sigma\in[\sigma_1,\sigma_2]$ and $u\in[x,N],$ provided $t\leq \delta x$ with $\delta$ small enough,
		$$|g(u)|\lesssim x^{-\veps},\ |g'(u)|\lesssim \sigma x^{-1-\veps}\leq x^{-\veps},\ |f'(u)|\leq \frac 12.$$
		Hence, 
		\begin{align*}
		D(s)&=\sum_{n=2}^{x}Q(n)(\log n)^\kappa (P(n))^{-s}+\int_x^N \frac{Q(u)\log^\kappa(u)}{(P(u))^s}du+\\
		&\quad\quad  \frac{b_d^{s-1}}{d^\kappa(s-1)^{\kappa+1}} \Gamma(\kappa+1,(s-1)\log(b_d P(N)))+O(x^{-\veps}+|s|N^{-\veps}+P(N)^{-\veps}).
		\end{align*}
		We let $N\rightarrow+\infty$, yielding to:
		\[D(s)=\sum_{n=2}^{x}Q(n)(\log n)^\kappa (P(n))^{-s}+\int_x^{+\infty} \frac{Q(u)\log^\kappa(u)}{(P(u))^s}du+O(x^{-\veps}),\qquad s\in\mathbb{C}_1.\]
		
		Now writing $Q(u)=P'(u)+Q_1(u)$ and repeating the argument in the first part of this proof one can obtain the following identity
		$$D(s)=\sum_{n=2}^{x}Q(n)(\log n)^\kappa (P(n))^{-s}+ \frac{b_d^{s-1}}{d^\kappa(s-1)^{\kappa+1}} \Gamma(\kappa+1,(s-1)\log(b_d P(x)))+\widetilde{R}(s),$$
		where $\widetilde{R}(s)$ is holomorphic in $\CC_{1-\frac 1d}$ and is $O(x^{-\varepsilon})$ in $\mathbb C_{\sigma_1}.$
		Using one last time \eqref{eq:ineggamma}, we obtain (b) of Lemma \ref{lem:analyticcontinuation}.
	\end{proof}

	From this and Lemma \ref{lem:MV}, we may deduce the first half of Theorem \ref{thm:main2}.
	\begin{proposition}\label{prop:pole1}
		Let $d\geq 1$, let $P\in\mathbb R[X]$ with $\deg(P)=d$ and $\lim_{+\infty}P=+\infty,$ let $Q\in\mathbb R[X]$ with $\deg(Q)=d-1$ and let $\kappa\in\mathbb R$.
		Assume moreover that $(\log(P(n)))$ is $\mathbb Q$-linearly independent.
		Then the Dirichlet series $D(s)=\sum_n Q(n)(\log n)^\kappa (P(n))^{-s}$ belongs to $\dwa(\sigma_0)$ with $\sigma_0=(2d-1)/2d.$
	\end{proposition}
	\begin{proof}
		It is clear that $\sigma_2(D)=(2d-1)/2d$ and thus it just remains to prove (d) of Definition \ref{def:dwa}. We fix $T\geq 1$ and we first estimate $\int_{T/2}^T |D(\sigma+it)|^2dt$
		where $1-\frac 1{2d}<\sigma_1\leq \sigma\leq \sigma_2.$
		We apply the estimate given by Lemma \ref{lem:analyticcontinuation} with $x=T/\delta$ so that 
		$O(x^{-\veps})=O(T^{-\veps})$ and 
		$$\left|\frac{\log^\kappa(P(x))}{(s-1)P(x)^{s-1}}\right|\lesssim  \frac{\log^\kappa T}{T T^{d(\sigma-1)}}\lesssim T^{-\veps}.$$
		Hence, applying Lemma \ref{lem:MV}
		\begin{align*}
		\int_{T/2}^T |D(\sigma+it)|^2dt&\lesssim \int_{T/2}^T \left|\sum_{n=2}^{T/\delta} |Q(n) (\log n)^\kappa  (P(n))^{-s}\right|^2 dt+T^{1-2\veps}\\
		&\lesssim T \sum_{n=2}^{T/\delta} |Q(n)|^2 (\log n)^{2\kappa} |P(n)|^{-2\sigma}+\\
		&\quad\quad \sum_{n=2}^{T/\delta}\frac{ |Q(n)|^2 (\log n)^{2\kappa} |P(n)|^{-2\sigma}}{\log(P(n+1))-\log(P(n))}+T^{1-2\veps}.
		\end{align*}
		The first sum is dominated by some constant since $\sigma\geq \sigma_1>\sigma_2(D).$ Regarding the second sum, for $n\in[2,T/\delta],$
		$$\frac{ |Q(n)|^2 (\log n)^{2\kappa} |P(n)|^{-2\sigma}}{\log(P(n+1))-\log(P(n))}\lesssim T n^{2(d-1)-2d\sigma}(\log n)^{2\kappa}\lesssim T n^{2d(1-\sigma_1)-2}(\log n)^{2\kappa},$$
		and we get the estimate
		$$\sum_{n=2}^{T/\delta}\frac{ |Q(n)|^2 (\log n)^{2\kappa} |P(n)|^{-2\sigma}}{\log(P(n+1))-\log(P(n))}\lesssim T,$$
		since $2d(1-\sigma_1)<1.$ 
		Hence, we have obtained 
		$$\int_{T/2}^T |D(\sigma+it)|^2 dt\lesssim T,$$
		for all $T\geq 1$ and all $\sigma\in[\sigma_1,\sigma_2],$ where the involved constant does not depend neither on $\sigma$ nor on $T.$
		Taking $T2^{-j}$ instead of $T$ in the latter formula and summing over $j$, we get the proposition.
	\end{proof}
	
	The second half of the proof of Theorem \ref{thm:main2} has been proven in \cite[Proposition 6.2]{Bayrearrangement}, for the sake of completeness, we repeat the argument below.
	
	\begin{proposition}\label{prop:pole2}
		Let $P\in\mathbb R_d[X]$ with $\lim_{+\infty}P=+\infty,$ let $Q\in\mathbb R_{d-1}[X]$ and let $\kappa\in\mathbb R$.
		Then the Dirichlet series $D(s)=\sum_n Q(n)(\log n)^\kappa (P(n))^{-s}$ belongs to $\ddens $.
	\end{proposition}
	
	\begin{proof}
		Let $\alpha,\,\beta>0$. There exists $x_0>1$ such that for every $x\geq x_0$ the polynomial $P$ is positive and increasing and $Q$ behaves like its leading term. 
		By Lemma~\ref{lemma61} there exist constants $c_1>0,\,c_2\in\mathbb{R}$ such that
		\[n\in \left[c_1e^{\frac{x}{d}}-c_2+\varepsilon,c_1e^{\frac{x}{d}+\frac{\alpha }{dx^2}}-c_2-\varepsilon \right]\qquad\text{implies that}\qquad \lambda_n\in\left[x,x+\frac{\alpha}{x^2}\right].\]
		Thus
		\[\sum\limits_{\lambda_n\in\left[x,x+\frac{\alpha}{x^2}\right]}\left|Q(n)(\log n)^\kappa\right|\gtrsim  \frac{e^{\frac{x}{d}}}{x^2}e^{\frac{d-1}{d}x(1-\frac{\beta}{2})}\gtrsim e^{(1-\beta)x} .\]
	\end{proof}
	
	Theorem \ref{thm:main2} now follows from Proposition \ref{prop:pole1}, Proposition \ref{prop:pole2} and Theorem \ref{thm:bayrearrangement}.
	
	%%%%%%%%%%%%%%%%%%%%%%%%%%%%%%%%%%%%%%%%%%%%
	%%%%%%%%%%%%%%%%%%%%%%%%%%%%%%%%%%%%%%%%%%%%
	%%%%%%%%%%%%%%%%%%%%%%%%%%%%%%%%%%%%%%%%%%%%
	
	\section{Random models and further discussion} \label{sec:randommodels}
	One of the motivations behind our work is to give concrete examples of convergent universal objects like the alternating prime zeta function $P_{-}(s)=\sum_{n\geq1}(-1)^np_n^{-s}$. As we have already proved $P_{-}$ is strongly universal in the critical strip. 
	It is worth mentioning that Theorem~\ref{thm:main} implies that every series of the form  
	\[P_{\chi}(s)=\sum_{n\geq1}\chi_n p_n^{-s},\qquad|\chi_n|=1,\]
	with $\sigma_c(P_{\chi})\leq 0$, is strongly universal in $\{\frac{1}{2}<\Re e<1\}$.
	
	Let us randomize our series. Let $X=(X_n)$ be a sequence of unimodular independent identically distributed Steinhaus or Rademacher (coin tossing) random variables and $P_{X}(s)=\sum_{n\geq1}X_n p_n^{-s}$. Kolmogorov’s three-series theorem \cite[Chapter~5]{MS13} implies that $P_{X}$ converges almost surely in $\mathbb{C}_{\frac{1}{2}}$. To obtain that such series are strongly universal almost surely, we need to obtain more information about their order in the critical strip.
	\begin{proposition}\label{prop:random}
		Let $P_{X}(s)=\sum_{n\geq1}X_n p_n^{-s}$, where $(X_n)$ is as above. Then, $P_{X}$ is of sub-logarithmic order in the critical strip and as a consequence is strongly universal, almost surely.
	\end{proposition}
	\begin{proof}
		We consider the corresponding randomized zeta functions
		\begin{equation}\label{eulerzeta}
		\zeta_X(s)=\prod_{n\geq 1}\frac{1}{1-X_np_n^{-s}}.
		\end{equation}
		It is easy to see that $\zeta_X$ converges absolutely for $\Re e s>1+\varepsilon,\,\varepsilon>0$. It is also known that $\zeta_X$ and the reciprocal $1/\zeta_X$  converge in $\mathbb{C}_{\frac{1}{2}}$, almost surely. For Steinhaus random variables $ (X_1,X_2,\dots)\in \mathbb{T}\times\mathbb{T}\times\dots$ this can be obtained from the work of Helson \cite{H70} or as an application of Menchoff's theorem \cite{BA02}. In the case of Rademacher random variables $X_n=r_n(t)= \sign(\sin(2\pi2^nt)),\, 0<t<1$ this has been done by Carlson and Wintner \cite{CARLSON,Wintner}.
		
		We set 
		$$F(s,X)=\sum_{n\geq 1}\sum_{k\geq 2}\frac{X_n^k}{k}p_n^{-ks}$$
		which is absolutely convergent in $\CC_{1/2}$ for all $X$. Starting from \eqref{eulerzeta} we get that
		\[\log\zeta_X-P_X=F\text{ in }\mathbb{C}_{1}.\]
		Using the Borel–Carathéodory theorem similarly as in the proof of the implication the Riemann hypothesis implies the Lindelöf hypothesis, \cite[Theorem~4.2]{TITCH}, we obtain that for all $X$ such that $\zeta_X$ and $1/\zeta_X$ both converge in $\CC_{1/2}$, for all $\varepsilon>0$
		\[\left|P_X(\sigma+it)\right|=O\left((\log t)^{2-2\sigma+\varepsilon}\right),\qquad t\rightarrow \infty,\]
		uniformly for $1\geq\sigma\geq\sigma_0>\frac{1}{2}$.
		
		We fix such an $X$. From Example \ref{ex:ddens} we get immediately that $P_{X}$ belongs to $\ddens$. It remains to show that $P_{X}\in \dwa(\frac{1}{2})$. We follow a method introduced in \cite{BaIv00} where the authors estimate the square moments of the logarithm of the zeta function. 
		
		Let $T>2$, let $\sigma_0>1/2$ and write it $\sigma_0=\frac12+2\delta$. Let $\veps>0$ and let us set $A=T^\varepsilon$.  
		The inverse Mellin transform (see \cite[Appendix 3]{MV}) applied to the $\Gamma$ function
		says that, for all $x>0,$
		\begin{equation}\label{eq:invmellin}
		e^{-x}=\frac1{2\pi i}\int_{\frac{3}{2}-\delta-i\infty}^{\frac{3}{2}-\delta+i\infty}x^{-w}\Gamma(w)dw. 
		\end{equation}
		We apply \eqref{eq:invmellin} for $x=\frac{p_n}{A}$, yielding to
		$$\exp\left(-\frac{p_n}{A}\right)=\frac 1{2\pi i}\int_{3/2-\delta-i\infty}^{3/2-\delta+i\infty} p_n^{-w}A^w\Gamma(w)dw.$$
		Therefore, for any $\sigma>\sigma_0$ and any $t\in[0,T],$ setting $s=\sigma+it,$ for any $n\geq 1,$ 
		$$X_{n} p_n^{-s}\exp\left(-\frac{p_n}{A}\right)=
		\frac 1{2\pi i}\int_{\frac{3}{2}-\delta-i\infty}^{\frac{3}{2}-\delta+i\infty}X_{n} p_n^{-s-w}A^w \Gamma(w)dw.$$
		Since $\Re e(s+w)>1$ provided $\Re e(w)=\frac{3}{2}-\delta$ we can sum these equalities and interchange summation and integral to get 
		\[\sum_nX_np_n^{-s}e^{-\frac{p_n}{A}}=\frac{1}{2\pi i}\int_{\frac{3}{2}-\delta-i\infty}^{\frac{3}{2}-\delta+i\infty}P_{X}(s+w)A^w\Gamma(w)\,dw,\]
		
		We introduce the following contour $\mathcal C=\cup_{i=1}^5\mathcal C_i$, defined as the union of five segments or half-lines.
		\begin{figure}[!ht]
			\begin{center}
				\begin{tikzpicture}[line cap=round,line join=round,>=triangle 45,x=0.7cm,y=0.7cm]
				\clip(-4.12,-6.74) rectangle (8.3,7.58);
				\draw[help lines, color=gray!30, dashed] (-4.9,-6.9) grid (6.9,6.9);
				\draw[->,ultra thick] (-5,0)--(7.5,0) node[right]{$x$};
				\draw[->,ultra thick] (0,-7)--(0,7) node[above]{$y$};
				\draw [line width=2.pt] (4.,3.) -- (4.,7.08);
				\draw [line width=2.pt] (4.,-3.) -- (4.,-6.74);
				\draw [line width=2.pt] (4.,3.)-- (-2.,3.);
				\draw [line width=2.pt] (-2.,3.)-- (-2.,-3.);
				\draw [line width=2.pt] (-2.,-3.)-- (4.,-3.);
				\draw [line width=2.pt] (4.,3.)-- (4.,3.843125);
				\draw (0.1,4) node[anchor=north east] {$\log^2(T)$};
				\draw (-1.94,0) node[anchor=north west] {$-\delta$};
				\draw (4.3,5.66) node[anchor=north west] {$\mathcal C_1$};
				\draw (4.14,-4.12) node[anchor=north west] {$\mathcal C_5$};
				\draw (0.82,2.98) node[anchor=north west] {$\mathcal C_2$};
				\draw (-2.88,1.34) node[anchor=north west] {$\mathcal C_3$};
				\draw (0.64,-3.) node[anchor=north west] {$\mathcal C_4$};
				\draw [line width=1.pt] (4,-0.1) -- (4,0.1);
				\draw (3.16,0.1) node[anchor=north west] {$\frac{3}{2}-\delta$};
				\end{tikzpicture}
			\end{center}
		\end{figure} 
		Stirling's formula for the $\Gamma$-function (see again \cite[Appendix 3]{MV}) says that 
		$$|\Gamma(u+iv)|\lesssim e^{-C|v|}$$
		for some $C>0$ (independent of $w=u+iv\in\mathcal C$). This implies that
		\[	\int_{\mathcal C_i}|P_{X}(s+w) A^w \Gamma(w)|\,dw\lesssim 1,\qquad i=1,\,2,\,4,\,5.\]
		To prove that the integral over the line segment $\mathcal C_3$ is bounded note that our function  $P_{X}$ is of zero order uniformly in $\mathbb{C}_{\frac{1}{2}+\delta}$. Since $|\Im m(s+w)|\lesssim T$ for $s=\sigma+it$ with $t\in[0,T]$ and $w\in \mathcal C_3$, we get
		\[	\int_{\mathcal C_3}|P_{X}(s+w) A^w \Gamma(w)|\,dw\lesssim T^{\frac{\varepsilon\delta}{2}-\varepsilon\delta}\lesssim 1.\]
		Let now $\mathcal R$ be the rectangle $\mathcal C_2\cup\mathcal C_3\cup\mathcal C_4\cup\mathcal C_6$ with $\mathcal C_6=[\frac{3}{2}-\delta-i\log^2T,\frac{3}{2}-\delta+i\log^2T]$
		so that 
		\begin{equation}\label{eq:lemint3}
		\int_{\mathcal C}=\int_{\frac{3}{2}-\delta-i\infty}^{\frac{3}{2}-\delta+i\infty}+\int_{\mathcal R}.
		\end{equation}
		The function $w\mapsto P_{X}(s+w)A^w\Gamma(w)$ has a single pole at $0$ in the interior of the curve $\mathcal R,$ 
		with residue $P_{X}(s).$ We apply Cauchy's theorem, yielding to:
		\[P_{X}(s)=\sum_nX_np_n^{-s}e^{-\frac{p_n}{A}}+O(1),\qquad \sigma>\sigma_0,\, t\in[0,T].\]
		By the Montgomery-Vaughan inequality, we obtain
		\begin{align*}
		\frac{1}{T}\int_0^{T}|P_{X}(\sigma+it)|^2dt&\lesssim
		T+\sum_{n\geq 1}p_n^{1-2\sigma}e^{-2\frac{p_n}A}\\
		&\lesssim T+\sum_{n\geq 1} n^{1-2\sigma} e^{-2\frac{n}{A}}\\
		&\lesssim T+ \sum_{n\geq A}n^{1-2\sigma}e^{-2\frac nA}+
		\sum_{n\geq A}n^{1-2\sigma}e^{-2\frac nA}\\
		&\lesssim T+A^{2-2\sigma}+\int_A^{+\infty}x^{1-2\sigma}e^{-2\frac xA}dx\\
		&\lesssim T+A^{2-2\sigma}\lesssim T
		\end{align*}
		if we choose $\veps<1/2$. Therefore $P_{X}\in \dwa(\frac{1}{2})$.
	\end{proof}
	
	\begin{question}
		Is it true that if the series $P_{X}$ converges in $\CC_{1/2}$, then it will be strongly universal in $\{\frac{1}{2}<\Re e<1\}$?
	\end{question}
	
	In view of the proof of Proposition \ref{prop:random}, this question is clearly linked to the order of the Dirichlet series in $\CC_{1/2}$.
	
	\begin{question}
		Let $\alpha>0$. What is the order of $\sum_{n\geq 1}(-1)^n p_n^{-s}$ or of a convergent $P_{X}$ in $\CC_\alpha$ ?
	\end{question}

	\bibliographystyle{amsplain} 
	\bibliography{ref} 
\end{document}